\documentclass[11pt]{article}

\usepackage[a4paper,margin=30mm]{geometry}
\usepackage{amsmath,amssymb,amsthm,mathtools}
\usepackage{microtype}
\usepackage{booktabs}
\usepackage{enumitem}
\usepackage{xcolor}
\usepackage{hyperref}
\usepackage[nameinlink,capitalize,noabbrev]{cleveref}

\hypersetup{
  colorlinks=true,
  linkcolor=blue!55!black,
  citecolor=blue!55!black,
  urlcolor=blue!55!black
}

\newcommand{\PP}{\mathbb P}
\newcommand{\kk}{\Bbbk}
\newcommand{\MM}{\mathfrak m}
\newcommand{\reg}{\operatorname{reg}}
\newcommand{\sat}{\operatorname{sat}}
\newcommand{\whalpha}{\widehat\alpha}

\theoremstyle{plain}
\newtheorem{theorem}{Theorem}[section]
\newtheorem{proposition}[theorem]{Proposition}
\newtheorem{lemma}[theorem]{Lemma}

\theoremstyle{definition}
\newtheorem{definition}[theorem]{Definition}

\theoremstyle{remark}
\newtheorem{remark}[theorem]{Remark}

\title{From Demailly's Inequality to\\ Harbourne--Huneke Containments for General Points}
\author{Grzegorz Malara}
\date{October 2026}

\begin{document}
\maketitle

\begin{abstract}
Let $I$ be the defining ideal of a general set of $s$ points in projective $N$-space over an algebraically closed field. Harbourne and Huneke conjectured the strengthened symbolic containment
$$
I^{(Nr)}\subseteq \MM^{r(N-1)}I^r
$$
for every $r\ge 1$. For general points this was known in low dimension and for small numbers of points, while in arbitrary dimension and arbitrary cardinality the best general result was the stable version, valid on one dense Zariski-open set for all sufficiently large $r$.

We show that the recent proof of Demailly's conjecture for arbitrary finite point sets by H\`a and Sivakumar removes the stability threshold. The key observation is that the strict generic Waldschmidt estimate of Bisui and Nguy\~en allows one to select a single finite symbolic level $m_0$. Equality of the corresponding interpolation number is a Zariski-open condition. Demailly's inequality then propagates this one finite condition to lower bounds for every symbolic power, and a standard regularity criterion converts those bounds into the Harbourne--Huneke containment for every $r$. Consequently, for every $N\ge2$ and every $s\ge1$, one dense Zariski-open family of $s$-point configurations in $\PP^N$ satisfies the Harbourne--Huneke containment simultaneously for all $r\ge1$.

The argument also clarifies a quantifier issue in the earlier literature: before the full Demailly theorem, estimates on the generic fibre naturally yielded either all $r$ on a very general set, or one open set only for $r\gg0$. The new theorem replaces infinitely many symbolic conditions by a single finite interpolation condition.
\end{abstract}

\section{Introduction}

Let $\kk$ be an algebraically closed field and let
$$
R=\kk[x_0,\ldots,x_N],\qquad \MM=(x_0,\ldots,x_N).
$$
If $X=\{P_1,\ldots,P_s\}\subset \PP^N_\kk$ is a finite set of distinct points, we write
$$
I=I(X)=I(P_1)\cap\cdots\cap I(P_s)
$$
for its radical homogeneous ideal. Its $m$-th symbolic power is
$$
I^{(m)}=I(P_1)^m\cap\cdots\cap I(P_s)^m,
$$
and $\alpha(J)$ denotes the initial degree of a nonzero homogeneous ideal $J$. The Waldschmidt constant is
$$
\whalpha(I)=\lim_{m\to\infty}\frac{\alpha(I^{(m)})}{m}
            =\inf_{m\ge1}\frac{\alpha(I^{(m)})}{m}.
$$

The uniform containment theorem of Ein--Lazarsfeld--Smith and Hochster--Huneke gives
$$
I^{(Nr)}\subseteq I^r
$$
for all $r\ge1$. Harbourne and Huneke proposed the stronger containment
\begin{equation}\tag{HH}\label{eq:HH}
I^{(Nr)}\subseteq \MM^{r(N-1)}I^r.
\end{equation}
which incorporates the extra homogeneous degree forced by the codimension of a point; see \cite{HH13}. This containment immediately implies Chudnovsky's lower bound for the Waldschmidt constant.

The history of \eqref{eq:HH} for general points is closely intertwined with Chudnovsky's and Demailly's conjectures. Harbourne--Huneke proved \eqref{eq:HH} for general points in $\PP^2$ \cite{HH13}; Dumnicki proved it for any number of general points in $\PP^3$, and also for at most $N+1$ points in general position in $\PP^N$ \cite{Dum12}. Nagel--Trok subsequently proved \eqref{eq:HH} for the remaining general-position configurations with at most $N+3$ points \cite[Proposition~5.4]{NT19}. For larger sets, Dumnicki--Tutaj-Gasi\'nska and later Bisui--Grifo--H\`a--Nguy\~en developed specialization and asymptotic methods, culminating in the theorem of Bisui--Nguy\~en that \eqref{eq:HH} holds \emph{stably} for any number of general points in any projective dimension \cite{DTG17,BGHN22,BN24}. More precisely, for fixed $N,s$ they obtain one dense open set and a number $r(N,s)$ such that \eqref{eq:HH} holds for every configuration in that open set and every $r\ge r(N,s)$ \cite[Theorem~5.2]{BN24}.

In August 2026, H\`a and Sivakumar proved Demailly's conjecture in full generality in characteristic zero \cite{HaSiv26} and in September 2020, Pokora and Szemberg completed the picture by providing a proof in finite characteristic \cite{PokSze26}: for every finite set of points in $\PP^N$ and every $m\ge1$,
\begin{equation}\tag{D}\label{eq:Demailly}
\whalpha(I)\ge
\frac{\alpha(I^{(m)})+N-1}{m+N-1}.
\end{equation}
Our purpose is to show that (D), combined with the strict generic Waldschmidt estimate already established by Bisui--Nguy\~en, upgrades the stable general-point result to all symbolic exponents.

The new part of the argument concerns the range
$$
N\ge4,\qquad s\ge N+4.
$$
In that range Bisui--Nguy\~en proved for the ideal $H$ of $s$ generic points that
\begin{equation}\label{eq:genericgap-intro}
N\whalpha(H)>\reg(H)+N-1.
\end{equation}
The strictness of this inequality is decisive. Since the finite Demailly quotients
$$
c_m(H)=\frac{\alpha(H^{(m)})+N-1}{m+N-1}
$$
converge to $\whalpha(H)$, one can choose a single $m_0$ for which
$$
Nc_{m_0}(H)>\reg(H)+N-1.
$$
The number $\alpha(H^{(m_0)})$ is finite interpolation data, hence it is preserved on one dense open family of specializations. Full Demailly then turns this one finite inequality into a lower bound for \emph{every} symbolic initial degree on every fibre in that open family. A standard regularity lemma yields \eqref{eq:HH} for all $r$.

Our main result is the following.

\begin{theorem}\label{thm:main}
Let $\kk$ be an algebraically closed field of arbitrary characteristic. For every $N\ge2$ and every $s\ge1$, there exists a nonempty dense Zariski-open subset $U$ of the parameter space of $s$ distinct points in $\PP^N_\kk$ such that for every $X\in U$, with defining ideal $I_X$, one has
$$
I_X^{(Nr)}\subseteq \MM^{r(N-1)}I_X^r
$$
for every integer $r\ge1$.
\end{theorem}

Thus the Harbourne--Huneke containment holds for a general set of any number of points, with ``general'' in the strong sense that a \emph{single} Zariski-open set works simultaneously for all $r$.

The point of the proof is not merely to remove finitely many exceptional values of $r$ from a stable theorem. It resolves a genuine quantifier obstruction. A generic-fibre argument can often prove the containment for every $r$ over the function field, and specialization then produces, for every fixed $r$, an open set $U_r$ on which that one containment holds. The intersection $\bigcap_r U_r$ describes a very general set, but need not contain a Zariski-open subset. The stable results overcome this problem by propagating one finite containment to all sufficiently large $r$. The full Demailly theorem offers a different propagation mechanism: one finite symbolic \emph{initial degree} controls all symbolic initial degrees, and therefore all $r$, on one open set.

\section{General, very general, and generic points}

We briefly fix the quantifiers used below. Let
$$
\mathcal P_{N,s}=\{(P_1,\ldots,P_s)\in(\PP^N)^s:\ P_i\ne P_j\text{ for }i\ne j\}.
$$
This is an irreducible open subset of $(\PP^N)^s$. Since all properties considered here are symmetric in the points, one may equivalently work on the corresponding open subset of the symmetric product or of the Hilbert scheme.

\begin{definition}
A property of $s$-point configurations holds for \emph{general points} if it holds on a nonempty dense Zariski-open subset of $\mathcal P_{N,s}$. It holds for \emph{very general points} if it holds on a countable intersection of nonempty dense Zariski-open subsets. The \emph{generic points} are the universal ordered configuration over the function field of $\mathcal P_{N,s}$.
\end{definition}

This distinction is emphasized in \cite{FMX18,BN24}. It is particularly relevant for symbolic powers. For a fixed exponent $r$, a rank condition on a finite interpolation matrix is open. Therefore a result proved for the generic fibre usually specializes to one open set $U_r$. Requiring the result for every $r$ leads a priori to the countable intersection
$$
\bigcap_{r\ge1}U_r,
$$
which is a very general condition rather than a general one.

The status before the proof of full Demailly can be summarized schematically as follows.

\begin{center}
\begin{tabular}{@{}p{0.29\textwidth}p{0.42\textwidth}p{0.20\textwidth}@{}}
\toprule
Range & Containment statement & Reference \\
\midrule
General points in $\PP^2$ & \eqref{eq:HH} for every $r$ & \cite{HH13} \\
General points in $\PP^3$ & \eqref{eq:HH} for every $r$; also shifted refinements & \cite{Dum12,Dum15} \\
At most $N+1$ points in general position & \eqref{eq:HH} for every $r$ & \cite{Dum12} \\
$N+2$ and $N+3$ general-position points & \eqref{eq:HH} for every $r$ & \cite{NT19} \\
Sufficiently many general points & one open set, \eqref{eq:HH} for $r\gg0$ & \cite{BGHN22} \\
Any number of general points & one open set, \eqref{eq:HH} for $r\ge r(N,s)$ & \cite{BN24} \\
\bottomrule
\end{tabular}
\end{center}

The stable theorem in the last row is deliberately formulated with a uniform threshold on one open set. Bisui--Nguy\~en explicitly distinguish this from the very-general situation and complete the stable picture for all $N$ and $s$ \cite[Section~5]{BN24}.

\section{Two finite ingredients}

We record the two elementary finite statements that allow us to turn the asymptotic inequality into an ideal containment.

\subsection{A regularity criterion}

The following lemma is standard in the containment literature. We state the version from \cite[Lemma~4.2]{DST13}.

\begin{lemma}[degree-to-containment criterion]\label{lem:regcriterion}
Let $I\subset R$ be a homogeneous ideal with zero-dimensional support. If $q,r,k$ are nonnegative integers and
$$
\alpha(I^{(q)})>r\,\reg(I)+k,
$$
then
$$
I^{(q)}\subseteq \MM^kI^r.
$$
\end{lemma}

\begin{proof}
For completeness, we recall the argument. Since
$$
\alpha(I^{(q)})\le q\alpha(I)\le q\reg(I),
$$
the assumed inequality forces $q\ge r$, hence $I^{(q)}\subseteq I^{(r)}=\sat(I^r)$. For ideals with zero-dimensional support one has
$$
\operatorname{satdeg}(I^r)\le\reg(I^r)\le r\reg(I).
$$
Thus $I^{(r)}$ and $I^r$ agree in degrees above $r\reg(I)$. Moreover $I^r$ is generated in degrees at most $r\reg(I)$. Consequently every element of $I^{(q)}$ occurs only in degrees greater than $r\reg(I)+k$ and, in those degrees, lies in $\MM^k I^r$.
\end{proof}

The strict inequality in Lemma~\ref{lem:regcriterion} is convenient for us because the generic estimate of Bisui--Nguy\~en is strict.

\subsection{One symbolic initial degree is an open condition}

Let $H$ denote the ideal of the generic ordered configuration of $s$ points in $\PP^N$. We use the following elementary specialization fact.

\begin{lemma}\label{lem:openalpha}
Fix an integer $m\ge1$. There is a nonempty dense Zariski-open subset $U_m\subset\mathcal P_{N,s}$ such that, for every $X\in U_m$,
$$
\alpha(I_X^{(m)})=\alpha(H^{(m)}).
$$
There is also a nonempty dense open subset $U_{\rm reg}$ on which
$$
\reg(I_X)=\reg(H).
$$
\end{lemma}

\begin{proof}
For each degree $d$, the vector space $(I_X^{(m)})_d$ is the kernel of the jet-evaluation map that sends a degree-$d$ form to its derivatives of order $<m$ at the $s$ points. After choosing affine charts locally on $\mathcal P_{N,s}$, this is a matrix whose entries are regular functions in the point coordinates. Its rank is maximal at the generic point and therefore constant on a nonempty open subset. Taking the finite intersection of these open sets for all degrees up to $\alpha(H^{(m)})$ gives the first assertion.

The second assertion follows in the same way from the generic Hilbert function, or from the usual semicontinuity of cohomology in the universal family of reduced points. Only finitely many degrees up to the generic regularity are involved.
\end{proof}

\begin{remark}\label{rem:noinfinite}
The finiteness in Lemma~\ref{lem:openalpha} is crucial. There is in general no reason that one can choose one open set on which $\alpha(I_X^{(m)})$ agrees with the generic value simultaneously for every $m$. Doing so by intersecting the individual open sets produces a very general condition. Our argument below deliberately chooses one symbolic level before specializing.
\end{remark}

\section{The Demailly upgrade principle}

For a homogeneous radical point ideal $J\subset R$, define
$$
c_m(J)=\frac{\alpha(J^{(m)})+N-1}{m+N-1}.
$$
Since $\alpha(J^{(m)})/m\to\whalpha(J)$, one has
$$
\lim_{m\to\infty}c_m(J)=\whalpha(J).
$$
The theorems of H\`a--Sivakumar and Pokora-Szemberg say that, over an algebraically closed field,
$$
c_m(J)\le\whalpha(J)
$$
for every $m\ge1$ \cite[Theorem~1.2]{HaSiv26}.

The following proposition isolates the mechanism that will be applied to general points.

\begin{proposition}[finite-open upgrade]\label{prop:upgrade}
Let $H$ be the defining ideal of $s$ generic points in $\PP^N$ over base field of arbitrary characteristic. Assume
\begin{equation}\label{eq:genericgap}
N\whalpha(H)>\reg(H)+N-1.
\end{equation}
Then there exists a nonempty dense Zariski-open subset $U\subset\mathcal P_{N,s}$ such that, for every $X\in U$ and every $r\ge1$,
$$
I_X^{(Nr)}\subseteq \MM^{r(N-1)}I_X^r.
$$
\end{proposition}

\begin{proof}
Set $R_0=\reg(H)$. Since $c_m(H)\to\whalpha(H)$ and the inequality in \eqref{eq:genericgap} is strict, there is an integer $m_0\ge1$ such that
\begin{equation}\label{eq:finitegap}
Nc_{m_0}(H)>R_0+N-1.
\end{equation}
By Lemma~\ref{lem:openalpha}, after replacing $\mathcal P_{N,s}$ by a nonempty dense open subset $U$, we may assume simultaneously that
\begin{equation}\label{eq:openconditions}
\alpha(I_X^{(m_0)})=\alpha(H^{(m_0)})
\quad\text{and}\quad
\reg(I_X)=R_0.
\end{equation}
for every $X\in U$. In particular,
$$
c_{m_0}(I_X)=c_{m_0}(H).
$$

Fix $X\in U$ and write $I=I_X$. The full Demailly inequality of H\`a--Sivakumar gives
\begin{equation}\label{eq:demaillyfiber}
\whalpha(I)\ge c_{m_0}(I)=c_{m_0}(H).
\end{equation}
Since $\whalpha(I)=\inf_q\alpha(I^{(q)})/q$, for every $r\ge1$ we have
$$
\alpha(I^{(Nr)})\ge Nr\whalpha(I)
\ge Nr c_{m_0}(H).
$$
Using \eqref{eq:finitegap} and \eqref{eq:openconditions}, this yields
$$
\alpha(I^{(Nr)})
>r(R_0+N-1)
=r\reg(I)+r(N-1).
$$
Applying Lemma~\ref{lem:regcriterion} with $q=Nr$ and $k=r(N-1)$ gives
$$
I^{(Nr)}\subseteq \MM^{r(N-1)}I^r.
$$
The open set $U$ was chosen before $r$ and is independent of it, so the containment holds there for all $r\ge1$ simultaneously.
\end{proof}

\begin{remark}\label{rem:roleDemailly}
For the generic ideal $H$ itself, the strict inequality \eqref{eq:genericgap} already gives
$$
\alpha(H^{(Nr)})\ge Nr\whalpha(H)>r\bigl(\reg(H)+N-1\bigr),
$$
and hence the desired containment for every $r$. Thus the new role of full Demailly is not to prove the all-$r$ statement on the generic fibre. Its role is to make the passage from the generic fibre to \emph{general} fibres finite: one fixed value $\alpha(H^{(m_0)})$ is enough to control every symbolic exponent after specialization.
\end{remark}

\section{General points in arbitrary dimension}

We now combine Proposition~\ref{prop:upgrade} with the strict Waldschmidt estimate of Bisui--Nguy\~en.

\begin{theorem}\label{thm:newrange}
Let $\kk$ be an algebraically closed field. Suppose
$$
N\ge4\qquad\text{and}\qquad s\ge N+4.
$$
Then a general set $X$ of $s$ points in $\PP^N_\kk$ satisfies
$$
I_X^{(Nr)}\subseteq \MM^{r(N-1)}I_X^r
$$
for every $r\ge1$, simultaneously on one dense Zariski-open set of configurations.
\end{theorem}

\begin{proof}
Let $H$ be the ideal of $s$ generic points. Bisui--Nguy\~en prove
$$
\whalpha(H)>\frac{\reg(H)+N-1}{N}
$$
for $N\ge4$ and $s\ge N+4$ \cite[Theorem~4.10]{BN24}. This is exactly the hypothesis of Proposition~\ref{prop:upgrade}.
\end{proof}

We can now complete the parameter range by invoking the previously known all-$r$ results.

\begin{proof}[Proof of Theorem~\ref{thm:main}]
If $N=2$, the assertion is the theorem of Harbourne--Huneke for general planar points \cite[Proposition~3.10]{HH13}. If $N=3$, it is due to Dumnicki \cite{Dum12}.

Assume $N\ge4$. For $s\le N+1$, the containment for points in general position was proved by Dumnicki \cite{Dum12}. For $s=N+2$ and $s=N+3$, it follows from Nagel--Trok \cite[Proposition~5.4]{NT19}; their hypotheses are satisfied by a general configuration. Finally, for $s\ge N+4$, apply Theorem~\ref{thm:newrange}.
\end{proof}

\begin{remark}
Nagel--Trok actually prove their small-cardinality result for a broader class of linearly general configurations and determine the relevant symbolic initial degrees explicitly \cite[Sections~3--5]{NT19}. Thus there is no hidden specialization issue in the cases $s=N+2,N+3$.
\end{remark}

\section{Why the stable theorem did not automatically give the all-$r$ theorem}

It is useful to isolate the change in logical structure brought by the full Demailly theorem.

Bisui--Nguy\~en already had the strict generic inequality
$$
N\whalpha(H)>\reg(H)+N-1.
$$
As observed in Remark~\ref{rem:roleDemailly}, this implies the Harbourne--Huneke containment for every $r$ for the generic ideal $H$. If one specializes this conclusion separately for each exponent, one obtains open sets
$$
U_1,U_2,U_3,\ldots
$$
such that the $r$-th containment holds on $U_r$. This yields the simultaneous result on
$$
\bigcap_{r\ge1}U_r,
$$
which is a very general subset. In general there is no algebraic-geometric mechanism forcing this countable intersection to contain a Zariski-open subset.

The stable containment results of \cite{BGHN22,BN24} overcome this obstacle in another way. They first specialize a carefully chosen finite containment and then use resurgence-type propagation to obtain \eqref{eq:HH} for all $r$ larger than a uniform threshold. Thus they produce one open set, but retain the condition $r\gg0$.

The proof of Proposition~\ref{prop:upgrade} uses the new Demailly theorem as a different propagation principle. The chain of implications is
$$
\begin{gathered}
N\whalpha(H)>\reg(H)+N-1
\\[2mm]
\Downarrow
\\[-1mm]
\text{choose one finite }m_0\text{ with }
Nc_{m_0}(H)>\reg(H)+N-1
\\[2mm]
\Downarrow
\\[-1mm]
\text{specialize this one finite interpolation number on one open set}
\\[2mm]
\Downarrow\quad\text{full Demailly on every fibre}
\\[-1mm]
\alpha(I_X^{(Nr)})>r\bigl(\reg(I_X)+N-1\bigr)
\quad\text{for every }r
\\[2mm]
\Downarrow
\\[-1mm]
I_X^{(Nr)}\subseteq \MM^{r(N-1)}I_X^r
\quad\text{for every }r.
\end{gathered}
$$
In quantifier notation, the generic-specialization route naturally gives
$$
\forall r\;\exists U_r,
$$
while Proposition~\ref{prop:upgrade} gives
$$
\exists U\;\forall r.
$$
The possibility of passing from one finite $m_0$ to all $r$ is exactly what became available with the proof of Demailly's conjecture for arbitrary point configurations.

\section{Further remarks}

\subsection{Relation to the historical direction of implication}

The direction of the argument is the reverse of the one that motivated much of the containment literature. Harbourne--Huneke introduced strengthened containments in part because taking initial degrees and passing to a limit yields Chudnovsky-type inequalities \cite{HH13}. Later work sought generalized containments strong enough to imply Demailly's conjecture; see, for instance, \cite{BGHN22D}. Now that Demailly's inequality is known independently for arbitrary finite point sets, it can be fed back into the containment problem, provided one has enough regularity margin.

The proof above shows precisely what is lost when one passes from an ideal containment to its initial-degree consequence: regularity is needed to recover the ideal-theoretic statement. For general points the strict inequality of Bisui--Nguy\~en supplies the necessary margin.

\subsection{A reusable criterion}

The argument is not intrinsically tied to generic points. If $I$ is any radical ideal of finitely many points and there exists $m\ge1$ such that
$$
N\frac{\alpha(I^{(m)})+N-1}{m+N-1}>\reg(I)+N-1,
$$
then Demailly's theorem and Lemma~\ref{lem:regcriterion} give
$$
I^{(Nr)}\subseteq\MM^{r(N-1)}I^r
$$
for every $r\ge1$. Thus a single finite symbolic interpolation computation can certify an infinite family of containments. We leave systematic applications to special configurations for separate work.

\subsection{Shifted containments}

The same method gives effective stable ranges for the stronger shifted containment
$$
I^{(Nr-N+1)}\subseteq \MM^{(r-1)(N-1)}I^r,
$$
but a positive finite-level margin does not by itself eliminate all small values of $r$. Bisui--Nguy\~en prove the shifted version stably for general points \cite[Remark~5.3]{BN24}, while Dumnicki proves the all-$r$ statement in $\PP^3$ \cite{Dum15}. Whether the present method, combined with sharper finite interpolation data, gives an all-$r$ shifted theorem in higher dimensions is a natural question.

\section*{Acknowledgement}
The Author is deeply grateful to Tomasz Szemberg for inspiring conversations and for invaluable comments provided at every stage of this work. Special thanks go to the authors whose work on generic and general points made the short argument above possible. Furthermore, the present manuscript was prompted by the recent proof of Demailly's conjecture in full generality by H\`a and Sivakumar in characteristic zero, and by Pokora and Szemberg in finite characteristic.

\end{document}